\documentclass[12pt,a4paper]{amsart}
\usepackage{amsfonts}
\usepackage{verbatim}
\newtheorem{theorem}{Theorem}[section]

\newtheorem{prop}{Proposition}[section]

\newcommand{\R}{\mathbb{R}}
\newcommand{\Rd}{\mathbb{R}^d}
\newcommand{\nN}{n \in \mathbb{N}}
\newcommand{\N}{\mathbb{N}}
\newcommand{\C}{\mathbb{C}}

\newcommand{\ds}{\displaystyle}

\newcommand{\tr}{\rm{tr}}

\newcommand{\bean}{\begin{eqnarray*}}
\newcommand{\eean}{\end{eqnarray*}}

\newcommand{\la}{\langle}
\newcommand{\ra}{\rangle}

\newcommand{\Z}{\mathbb{Z}}

\newcommand{\Q}{\mathbb{Q}}
\newcommand{\A}{\mathbb{A}}
\newcommand{\J}{\mathbb{J}}

\newcommand{\G}{\widehat{G}}

\date{}

\begin{document}

\title[Probabilistic Trace and Poisson Summation Formulae]{Probabilistic Trace and Poisson Summation Formulae on Locally Compact Abelian Groups}

\author{David Applebaum}

\address{School of Mathematics and Statistics, University of
Sheffield, Hicks Building, Hounsfield Road, Sheffield,
England, S3 7RH}

\email{D.Applebaum@sheffield.ac.uk }

\begin{abstract}We investigate convolution semigroups of probability measures with continuous densities on locally compact abelian groups, which have a discrete subgroup such that the factor group is compact. Two interesting examples of the quotient structure are the $d$--dimensional torus, and the ad\`{e}lic circle. Our main result is to show that the Poisson summation formula for the density can be interpreted as a probabilistic trace formula, linking values of the density on the factor group to the trace of the associated semigroup on $L^{2}$-space. The Gaussian is a very important example. For rotationally invariant $\alpha$-stable densities, the trace formula is valid, but we cannot verify the Poisson summation formula. To prepare to study semistable laws on the ad\`{e}les, we first investigate these on the $p$--adics, where we show they have continuous densities which may be represented as series expansions. We use these laws to construct a convolution semigroup on the ad\`{e}les whose densities fail to satisfy the probabilistic trace formula.

\vspace{5pt}


\vspace{5pt}

\end{abstract}

\maketitle

\section{Introduction}

The classical Poisson summation formula is a well-known result from elementary Fourier analysis. It states that for a suitably well--behaved function $f:\R \rightarrow \C$ (and typically $f$ is in the Schwartz space of rapidly decreasing functions), we have
\begin{equation} \label{PSi1}
\sum_{n \in \Z}f(n) = \sum_{n \in \Z}\widehat{f}(n),
\end{equation}
where $\widehat{f}$ is the Fourier transform of $f$ (see e.g. \cite{SS} p.154-6 or \cite{Katz} p.161). 
The importance of (\ref{PSi1}) can be seen from the fact that if $f$ is taken to be a suitable Gaussian, then (\ref{PSi1}) yields the celebrated functional equation for Jacobi's theta function, which was brilliantly utilised by Riemann to establish the functional equation for the zeta function; this is itself then applied to analytically continue that function to a meromorphic function on the complex plane (see e.g. the very accessible account in \cite{TMF}).

The Poisson summation formula involves wrapping a function around the torus $\mathbb{T} = \R/\Z$, and so it is natural to generalise it to the case where we have a discrete subgroup $\Gamma$ of a locally compact abelian group $G$ such that the factor group $G/\Gamma$ is compact. In this case we obtain
\begin{equation} \label{PSi2}
\sum_{\gamma \in \Gamma}f(n) = \sum_{\chi \in \widehat{G/\Gamma}}\widehat{f}(\chi),
\end{equation}
where $\widehat{G/\Gamma}$ is the dual group of $G/\Gamma$. This general case turns out to be a very useful tool in algebraic number theory. In his 1950 thesis, Tate \cite{Tate} used a slight extension of (\ref{PSi2}), which he called the ``Riemann--Roch theorem'' to establish the analytic continuation of local zeta functions by using harmonic analytic techniques. In that case, $G$ is the ad\`{e}les group, to be denoted $\mathbb{A}$, $\Gamma$ is the rational numbers and $G/\Gamma$ is the so-called ``ad\`{e}lic circle''. In fact the classical Riemann--Roch theorem of algebraic geometry, for curves over finite fields, may itself be derived from Tate's formula (see section 7.2. of \cite{RV}). Its worth pointing out that the ad\`{e}les group is a very rich mathematical structure that plays an important role in the Langlands programme \cite{Gel}; it is also central to attempts to solve the Riemann hypothesis using non--commutative geometry \cite{CCM}, and has been used to model string amplitudes in physics \cite{BrFr}. A key feature of the ad\`{e}les that makes them so mathematically attractive is that they put all the different completions of the rational numbers that are induced by a norm, on an equal footing, i.e. the real numbers together with the collection of non-Archimedean $p$-adic number fields, where $p$ varies over the set of prime numbers.

The main purpose of this paper is to explore the interaction of the abstract Poisson summation formula (\ref{PSi2}) with probability theory. For this we need some probabilistic input, and that is provided by a convolution semigroup of probability measures $(\mu_{t}, t \geq 0)$ on $G$, which is such that for all $t > 0, \mu_{t}$ has a continuous density $f_{t}$. This is a very rich class of measures. On $\R^{d}$ it includes Gauss--Weierstrass heat kernels, and also the $\alpha$--stable laws which have a number of desirable properties, such as self--similarity and regularly varying tails. Our main result is to show that, when we consider it within the context of such convolution semigroups, (\ref{PSi2}) may be interpreted as a {\it probabilistic trace formula}:
\begin{equation} \label{PSi3}
F_{t}(e) = \tr(P_{t}),
\end{equation}
for each $t > 0$, where $F_{t}$ is the pushforward of $f_{t}$ to the density of a convolution semigroup on $G/\Gamma$, and $P_{t}$ is the associated Markov semigroup of trace--class operators acting in $L^{2}(G/\Gamma)$. Formulae of the form (\ref{PSi3}) have been established for a class of central convolution semigroups on compact Lie groups in \cite{App3}. Such formulae also arise in the study of heat kernels on compact Riemannian manifolds (see Corllary 3.2 on p.90) of \cite{Rose}). In fact, all of these results can be seen as special cases of Mercer's theorem in functional analysis (see e.g. \cite{EBD1}, Proposition 5.6.9., p.156).  In the classical case (\ref{PSi1}), we investigate examples of the probabilistic trace formula arising from both Gaussian and rotationally invariant $\alpha$-stable convolution semigroups. Only in the Gaussian case can we legitimately realise this as an instance of the Poisson summation formula, and it is conceivable that this is, in fact, the only convolution semigroup where that formula is valid.

In the second part of the paper, our goal is to begin the probabilistic investigation of the ``Riemann-Roch form'' of the Poisson summation formula. Here we can report only limited progress. We must find good examples of convolution semigroups on the ad\`{e}le group $\mathbb{A}$. We know of no previous probabilistic work in this context other than \cite{KaVM} and \cite{Yas,Urb,Yas4}, where the approach seems to be different. Since the group $\mathbb{A}$ is a restricted direct product of the additive group $\Q_{p}$ of $p$-adic rationals, where $p$ ranges over the set of prime numbers (including $p=\infty$, for the real numbers), we first turn our attention to the study of convolution semigroups on $\Q_{p}$. At this stage we should emphasise that we are working in the context of ``classical probability'' on $p$-adic groups, and not the $p$-adic probability theory of \cite{Kh}, where the probabilities are themselves $p$-adic numbers. We investigate a $\gamma$-semistable rotationally invariant convolution semigroup $(\rho_{t}^{(p)}, t \geq 0)$ that has been studied by a number of authors; see \cite{AK1,AK2,Yas1} for its construction as a Markov process by solving Kolmogorov's equations, \cite{Koch, Yas3, Yas2} for characterisation as a limit of sums of i.i.d. $p$-adic random variables, \cite{AE} for a construction using Dirichlet forms,  and \cite{HS,DMFT} for a Fourier analytic approach within the spirit of the current article. An account of the relationship between some of these different constructions may be found in \cite{AZ}. Here there is already an interesting departure from the Euclidean case, where the index of stability $\alpha$ is constrained by the restriction $0  < \alpha \leq 2$. The $p$-adic counterpart has index of stability $\gamma$ which can take any positive value. Our main result here is to show that, just as in Euclidean case, $\rho_{t}^{(p)}$ has a density $f_{t}^{(p)}$ for each $t > 0$. We also find a series expansion for $f_{t}^{(p)}$. In the final part of this section, we show how to use the $p$-adic distributions described above to construct a convolution semigroup on the ad\`{e}les, and we conclude with a calculation that demonstrates that its projection to the ad\`{e}lic circle fails to satisfy the probabilistic trace formula.

Our investigations yield three different types of behaviour for convolution semigroups of measures. Gaussian densities in Euclidean space satisfy the Poisson summation formula, which is a special case of the probabilistic trace formula. Stable densities satisfy the second of these, but not necessarily the first. Finally the densities we have constructed on the ad\`{e}les satisfy neither of these.

It is well--known that the Poisson summation formula is a special case of the Selberg trace formula with compact quotient (see e.g. \cite{Arth} p.9). Gangolli \cite{Gang} has given the latter a probabilistic interpretation by using wrapped heat kernels on symmetric spaces (see in particular, Proposition 4.6 therein). More recently, Albeverio et al. \cite{AKY} have established a $p$-adic analogue of the Selberg formula using the semistable process described above.  It would be interesting to extend their $p$-adic formula to the ad\`{e}les, and also to seek a more general probabilistic trace formula which subsumes all the cases discussed here.

\vspace{5pt}

{\it Notation and Basic Concepts.} If $X$ is a topological space, then ${\mathcal B}(X)$ is the Borel $\sigma$--algebra of $X$. All measures introduced in this paper will be assumed to be defined on $(X, {\mathcal B}(X))$, and to be regular (i.e. they are both inner and outer regular, and every compact subset of $X$ has finite mass). Note that this is the case for Haar measure on a locally compact abelian group.

If $X$ is locally compact, then $C_{c}(X)$ will denote the space of all continuous, real--valued functions defined on $X$ that have compact support. By the Riesz representation theorem, regular measures on $(X, {\mathcal B}(X))$ are in one-to-one correspondence with positive, linear functionals on $C_{c}(X)$.

Let $G$ be a locally compact, Hausdorff abelian group with neutral element $e$. We will write the group law in $G$ multiplicatively, except in examples where addition is natural. The Dirac mass at $e$ will be denoted $\delta_{e}$. If $\mu_{1}$ and $\mu_{2}$ are probability measures on $G$, their {\it convolution} is the unique probability measure $\mu_{1} * \mu_{2}$ such that for all $h \in C_{c}(X)$,
$$ \int_{G}h(x)(\mu_{1} * \mu_{2})(dx) = \int_{G}\int_{G}h(xy)\mu_{1}(dx)\mu_{2}(dy).$$
We say that a family $(\mu_{t}, t \geq 0)$ of probability measures on $G$ is an {\it i-convolution semigroup} if
\begin{enumerate}
\item $\mu_{s} * \mu_{t} = \mu_{s+t}$ for all $s,t \geq 0$,
\item $\lim_{t \rightarrow 0}\int_{G}h(x)\mu_{t}(dx) = \int_{G}h(x)\mu_{0}(dx)$, for all $h \in C_{c}(G)$.
\end{enumerate}
Note that $\mu_{0}$ is an idempotent measure, i.e. $\mu_{0} * \mu_{0} = \mu_{0}$, and the prefix ``i-'' is meant to indicate this. In the case where $\mu_{0} = \delta_{e}$, we say that we have a {\it convolution semigroup} of probability measures on $G$.
We say that a probability measure $\rho$ on $G$ is embeddable, if there exists an i-convolution semigroup $(\mu_{t}, t \geq 0)$ on $G$ with $\rho = \mu_{1}$. If $G$ is arcwise connected, then every infinitely divisible measure on $G$ is embeddable (see Theorem 4.5 in \cite{Hey1}).

If $\mu$ is a measure on $G$ and $f:G \rightarrow G$ is measurable, then $f\mu$ is the (pushforward) measure on $G$ given by $(f\mu)(A):= \mu(f^{-1}(A))$, for all $A \in {\mathcal B}(G)$.

If $x = (x_{1}, \ldots, x_{d}) \in \R^{d}$, then the Euclidean norm of $x$ is $|x| = \left(\sum_{i=1}^{d}x_{i}^{2}\right)^{1/2}$.

\section{The Poisson Summation Formula}

Let $G$ be a locally compact, Hausdorff abelian group, and $\Gamma$ be a discrete subgroup of $G$ such that the factor group of left cosets, $G/\Gamma$, is compact. If $x \in G$ we often use the notation $[x]$ for the coset $x\Gamma$. The quotient map from $G$ to $G/\Gamma$ will be denoted by $\pi$. It is a continuous homomorphism.  Given a Haar measure on $G$, which we write informally as $dx$, and fixing counting measure (which is a Haar measure) on $\Gamma$, then it is well-known that there exists a Haar measure $d[x]$ on $G/\Gamma$ so that for all  $h \in L^{1}(G)$ we have
\begin{equation} \label{intform}
\int_{G}h(x)dx = \int_{G/\Gamma}\sum_{\gamma \in \Gamma}h(x\gamma)d[x],
\end{equation}
(see e.g. Reiter \cite{Rei}, p.69-71).

Let $\widehat{G}$ be the dual group of $G$. Then $\G$ is itself a locally compact abelian group, when equipped with the compact--open topology, which we may also equip with a Haar measure. The Fourier transform of $h \in L^{1}(G)$ is $\widehat{h}: \G \rightarrow \C$ which is given by
\begin{equation} \label{FT}
\widehat{h}(\chi) = \int_{G}h(x)\chi(x^{-1})dx,
\end{equation}
for all $\chi \in \G$. If $h \in L^{1}(G)$ is continuous and  $\widehat{h} \in L^{1}(\G)$, then Fourier inversion holds, and for all $x \in G$, 
\begin{equation} \label{FI}
h(x) = \int_{\G}\widehat{h}(\chi)\chi(x)d\chi.
\end{equation}
For details see e.g. Folland \cite{Foll} Theorem 4.3.3, p.111.
If $G$ is compact, then $\G$ is discrete, and (\ref{FI}) is a sum. In particular, it is known that $\widehat{G/\Gamma} \simeq \Gamma^{\perp}:= \{\chi \in \G; \chi(\gamma) = 1~\mbox{for all}~\gamma \in \Gamma\}$, and we will find it convenient to identify these two groups in the sequel.

From now on, we assume that $h \in L^{1}(G)$ is such that:

\begin{enumerate}
\item[(A1).] The series $\sum_{\gamma \in \Gamma}h(x\gamma)$ converges absolutely and uniformly for all $x \in G$,

\item[(A2).] The series $\sum_{\chi \in \Gamma^{\perp}}\widehat{h}(\chi)$ converges absolutely.

\end{enumerate}

The following is based on Lang \cite{La} pp.291-2. We include it for the convenience of the reader, noting that we will make use of some aspects of the proof within the sequel.

\begin{theorem} \label{PSF} [The Poisson Summation Formula].
If $h:G \rightarrow \C$ is continuous and integrable, then for all $x \in G$
\begin{equation} \label{PSF1}
\sum_{\gamma \in \Gamma}h(x\gamma) = \sum_{\chi \in \Gamma^{\perp}}\widehat{h}(\chi)\chi([x]).
\end{equation}
In particular,
\begin{equation} \label{PSF2}
\sum_{\gamma \in \Gamma}h(\gamma) = \sum_{\chi \in \Gamma^{\perp}}\widehat{h}(\chi).
\end{equation}
\end{theorem}

\begin{proof} Define $H:G/\Gamma \rightarrow \C$ by
$$ H([x]) = \sum_{\gamma \in \Gamma}h(x\gamma),$$
for all $x \in G$. By uniformity of convergence of the series, $H$ is continuous at $[x]$, and so by Fourier inversion (\ref{FI}),
$$ H([x]) = \sum_{\chi \in \widehat{G/\Gamma}}\widehat{H}(\chi)\chi([x]) = \sum_{\chi \in \Gamma^{\perp}}\widehat{H}(\chi)\chi(x).$$
Now using (\ref{FT}) and (\ref{intform}) we deduce that for all $\chi \in \Gamma^{\perp}$
\bean \widehat{H}(\chi) & = & \int_{G/\Gamma}H([x])\chi([x^{-1}])d[x]\\
& = & \int_{G/\Gamma}\sum_{\gamma \in \Gamma}h(x\gamma)\chi(\gamma^{-1}x^{-1})d[x]\\
& = & \int_{G}h(x)\chi(x^{-1})dx = \widehat{h}(\chi), \eean
and the result follows. \end{proof}

\section{A Probabilistic Trace Formula}

Let $f: G \rightarrow \R$ be a probability density function (pdf, for short), so that $f \geq 0$ and $\int_{G}f(g)dg = 1$. We define its {\it $\Gamma$ - periodisation} to be the mapping $F: G/\Gamma \rightarrow \R$ given for all $x \in G$  by
\begin{equation} \label{per}
F([x]) = \sum_{\gamma \in \Gamma} f(x\gamma),
\end{equation}

\noindent provided that the series converges for all $x \in G$.

It follows from (\ref{intform}) that $F$ is also a pdf. In fact we can say a little more. Recall that if $\mu$ is a (finite) measure on $(G, {\mathcal B}(G))$, then $\tilde{\mu}: = \mu \circ \pi^{-1}$ is a (finite) measure on $(G/\Gamma, {\mathcal B}(G/\Gamma))$. Moreover, if $h$ is a mapping from $G/\Gamma$ to $\C$, then $h_{\pi}: = h \circ \pi$ is a $\Gamma$-invariant mapping from $G$ to $\C$, which is continuous if $h$ is. Now suppose that $\mu$ is a probability measure on $G$ having a pdf $f$. If $\mu$ is $\Gamma$-invariant, then so is $f$ and it is easy to see that $\tilde{\mu}$ is absolutely continuous with respect to the induced Haar measure, with density $\tilde{f}(x\Gamma): = f(x)$, for all $x \in G$. Clearly if $f$ is $\Gamma$-invariant, then $\sum_{\gamma \in \Gamma}f(x\gamma)$ diverges. On the other hand, we have:

\begin{prop} \label{dchoice} If $\mu$ has a density $f$  and the series $\sum_{\gamma \in \Gamma}f(\cdot\gamma)$ converges to an $L^{1}$-function $\tilde{f}$, then $\tilde{f}$ is the density of $\tilde{\mu}$. 
\end{prop}

\begin{proof} Let $h \in C_{c}(G/\Gamma)$. Then by (\ref{intform}),
\bean \int_{G/\Gamma}h([x])\tilde{\mu}(d[x]) & = & \int_{G}h_{\pi}(x)\mu(dx)\\
& = & \int_{G}h_{\pi}(x)f(x)dx\\
& = & \int_{G/\Gamma}\sum_{\gamma \in \Gamma}h_{\pi}(x \gamma) f(x \gamma)d[x]\\
& = & \int_{G/\Gamma}h([x])\sum_{\gamma \in \Gamma} f(x \gamma)d[x],\eean
and the result follows. \end{proof}


If $\mu$ is a bounded Borel measure on $G$, its Fourier transform is the mapping $\widehat{\mu}: \G \rightarrow \G$ defined by $\widehat{\mu}(\chi) = \int_{G}\chi(g^{-1})\mu(dg)$ for all $\chi \in \G$.

The next result is quite well-known. We include it for completeness.
\begin{prop} \label{Fknown}
If $\sum_{\chi \in \Gamma^{\perp}}|\widehat{\mu}(\chi)| < \infty$, then $\tilde{\mu}$ has a continuous density.
\end{prop}

\begin{proof} For each $x \in G$, define $F([x]) = \sum_{\chi \in \Gamma^{\perp}} \chi(x)\widehat{\mu}(\chi)$. Then  $F$ is well-defined and continuous, by the dominated convergence theorem. Since $\widehat{F} = \widehat{\tilde{\mu}}$ the result follows by Fourier inversion.
\end{proof}

Now let $(\mu_{t}, t \geq 0)$ be a convolution semigroup of probability measures on $G$. We assume that $\mu_{t}$ has a continuous density $f_{t}$ (with respect to Haar measure) for all $t > 0$.
In particular $\mu_{t}$ is infinitely divisible for all $t \geq 0$ and there exists a continuous negative definite function $\eta: \G \rightarrow \C$ so that for all $t \geq 0, \chi \in \G$,
\begin{equation} \label{LK}
\widehat{\mu_{t}}(\chi) = e^{- t \eta(\chi)}.
\end{equation}

In fact a great deal is known about the structure of $\eta$ -- it is described by an abstract version of the classical L\'{e}vy--Khintchine formula. We will not need this level of detail, which can be found in Chapter IV, section 7 of \cite{Par} (see also Theorem 8.3 of \cite{BF}).  It is well--known (and easily checked) that $(\widetilde{\mu_{t}}, t \geq 0)$ is a convolution semigroup of probability measures on $G/\Gamma$, where $\widetilde{\mu_{t}}: = \mu_{t} \circ \pi^{-1}$. We introduce the $C_{0}$--semigroup $(P_{t}, t \geq 0)$ of positivity preserving operators $P_{t}:L^{2}(G/ \Gamma) \rightarrow L^{2}(G/ \Gamma)$ (see Chapter 5 of \cite{App} for details) given by
\begin{equation} \label{sgroup}
P_{t}H([x]) = \int_{G/\Gamma}H([xy])F_{t}([y])d[y]
\end{equation}
for each $t \geq 0, H \in L^{2}(G/ \Gamma), x \in G$.

Now note that $\{\overline{\chi}, \chi \in \Gamma^{\perp}\}$ is a complete orthonormal basis for $L^{2}(G/ \Gamma)$. For each $t \geq 0, \chi \in \Gamma^{\perp}, x \in G$ we have
\bean P_{t}\overline{\chi}([x]) & = & \int_{G/\Gamma}\overline{\chi}([xy])F_{t}([y])d[y]\\
& = & \widehat{F_{t}}(\chi)\overline{\chi}([x])\\
& = & \widehat{f_{t}}(\chi)\overline{\chi}([x])\\
& = & e^{-t \eta(\chi)}\overline{\chi}([x]), \eean
by (\ref{LK}), noting that $\widehat{F_{t}}(\chi) = \widehat{f_{t}}(\chi)$ was established within the proof of Theorem \ref{PSF}. So we have deduced that $\{e^{-t \eta(\chi)}, \chi \in \Gamma^{\perp}\}$ is the spectrum of the operator $P_{t}$.

Then (\ref{PSF1}) yields

\begin{equation} \label{PTF1}
 F_{t}([x^{-1}] = \sum_{\chi \in \Gamma^{\perp}}e^{-t \eta(\chi)}\chi([x]),
 \end{equation}

\noindent and when we take $x = e$, we obtain

\begin{equation} \label{PTF2}
 F_{t}(e) = \tr(P_{t}).
 \end{equation}

(\ref{PTF2}) is our required {\it probabilistic trace formula}. The formula (\ref{PTF2}) holds more generally for densities of certain central convolution semigroups on non-abelian compact Lie groups -- see Theorem 5.3 of \cite{App3}. Note that it also holds for i-convolution semigroups, as we make no use of the measure $\mu_{0}$ in the above argument. Furthermore, if  (\ref{PTF2}) fails to hold, then (\ref{PSF2}) cannot be valid.

When we proved Theorem \ref{PSF}, we required both assumptions (A1) and (A2) to hold. In the case where $f_{t}$ is the continuous density of a measure $\mu_{t}$ from a convolution semigroup, we only require that (A1) holds. To see why this is true, suppose that only (A1) holds. 
Then $F_{t}$ is continuous on $G/\Gamma$, and hence square--integrable. It then follows from Theorem 5.4.4 of \cite{App} (see also section 3 of \cite{App2}) that $P_{t}$ is trace--class.\footnote{Although the results in \cite{App} and \cite{App2} were proved for compact Lie groups, the Lie structure is not needed for them to hold.} We give the essential details. We first rewrite (\ref{sgroup}) as
$$ P_{t}H([x]) = \int_{G/\Gamma}H([y])F_{t}([x^{-1}y])d[y].$$
Then $P_{t}$ is a bounded linear operator with a continuous kernel, which (since $G/\Gamma$ is compact) is also square-integrable. Thus $P_{t}$ is Hilbert--Schmidt for all $t > 0$. But now $P_{t} = P_{t/2}P_{t/2}$ is the product of two Hilbert--Schmidt operators, and so it is trace--class. Hence (A2) is a consequence of (A1).

We now obtain an immediate consequence of (\ref{PTF2}). If $(\mu_{t}, t \geq 0)$ is symmetric for all $t \geq 0$, then $F_{t}([x]) = F_{t}([x^{-1}])$ for all $t \geq 0, x \in G$. In this case $\eta(\chi) \geq 0$ for all $\chi \in \G$, and from (\ref{PTF1}) and (\ref{PTF2}), we easily deduce that for all $x \in G$;
$$ F_{t}([x]) \leq F_{t}(e).$$

To be precise, we have  \bean F_{t}([x]) & = & \sum_{\chi \in \Gamma^{\perp}}e^{-t \eta(\chi)}\chi([x])\\ & = & \sum_{\chi \in \Gamma^{\perp}}e^{-t \eta(\chi)}\overline{\chi([x])}, \eean and so
$$ F_{t}([x]) = \sum_{\chi \in \Gamma^{\perp}}e^{-t \eta(\chi)}\Re(\chi([x])),$$
and the inequality follows from the fact that $|\Re(\chi([x]))| \leq 1$ for all $\chi \in \Gamma^{\perp}$.

Suppose that $G/\Gamma$ is metrisable (which is always the case when it is second countable). Then for all $t > 0, \mu_{t}$ is recurrent, and so the potential measure of every open ball in $G/\Gamma$ is infinite (see Theorem 4.1 and Proposition 4.1 in \cite{App4}). Consequently the integral $\int_{0}^{\infty}F_{t}(\cdot)dt$, which would define the potential density in the transient case, is also infinite. This last fact can also be seen by formally integrating both sides of (\ref{PTF1}) and observing that $\Gamma^{\perp}$ contains the trivial character $\chi_{0}$ which maps each point in $G$ to $1$, but $\eta(\chi_{0}) = 0$.

On the other hand, let $\Gamma_{0}^{\perp}: = \Gamma^{\perp} \setminus \{\chi_{0}\}$. Then provided that the sum on the right hand side makes sense, there may be a possibility for giving meaning to the identity obtained by formally integrating (\ref{PTF2}) after $1$ has been subtracted from both sides, to obtain:

\begin{equation} \label{intr}
\int_{0}^{\infty}(F_{t}(e) - 1)dt = \sum_{\chi \in \Gamma^{\perp}_{0}}\frac{1}{\eta(\chi)}.
\end{equation}

In section 4, we will examine examples where this can be done, and also where it can't. The left hand side of (\ref{intr}) is a potential theoretic quantity. The right hand side may be thought of as the value at $z = 2$  of a generalised zeta-function $\sum_{\chi \in \Gamma^{\perp}_{0}}1/\eta(\chi)^{z/2}$ (see also the discussion at the top of p.225 in \cite{App2}).

Note that if $\mu_{t}$ is symmetric for all $t \geq 0$, i.e. $\mu_{t}(A) = \mu_{t}(A^{-1})$ for all $A \in \mathcal{B}(G)$, then $\widehat{\mu_{t}}(\chi) \in \R$ for all $\chi \in \widehat{G}$, and $\eta(\chi) > 0$ for all $\chi \in \Gamma^{\perp}_{0}$. In this case (\ref{intr}) holds if and only if $\sum_{\chi \in \Gamma^{\perp}_{0}}1/\eta(\chi)$ converges.


In all the examples that we consider in sections 5 and 6, it will always be the case that the density $f_{t}$, for $t > 0$, is not $\Gamma$-invariant. This is easily verified, either directly from the given form of $f_{t}$ or from inspection of the Fourier transform, using the fact that $f_{t}$ is $\Gamma$-invariant, if and only if $\widehat{f_{t}}(\chi) = \widehat{f_{t}}(\chi)\chi(\gamma)$ for all $\chi \in \widehat{\Gamma}, \gamma \in \Gamma$.

\section{The Connection with Induced Representations}

If $h \in L^{1}(G/\Gamma)$ and $\pi$ is a unitary representation of $G$ in some complex, separable Hilbert space $H$, we may define the {\it co-Fourier transform} of $h$ at $\pi$ to be the operator
$$ \pi(h) = \int_{G/\Gamma}h(x)\pi(x)dx,$$
which is well--defined as a Pettis integral in $H$.

As pointed out by Mackey in \cite{Mac}, if we take $H = L^{2}(G/\Gamma)$ and $\pi$ to be the left action of $G$ on $H$ defined by
$$ \pi(x)F([y]) = F([x^{-1}y]),$$ for all $x,y \in G, F \in H$,
then $\pi$ is the induced representation obtained from the trivial representation of $\Gamma$ on $\C$.

\begin{prop} For all $t > 0$, 
$$ \tr(\pi(F_{t})) = \tr(P_{t}).$$
\end{prop}

\begin{proof} By Fubini's theorem, \bean \tr(\widehat{\pi(F_{t})}) & = & \sum_{\chi \in \Gamma^{\perp}} \la \pi(F_{t}) \chi, \chi \ra \\
& = & \sum_{\chi \in \Gamma^{\perp}} \int_{G/\Gamma} \int_{G/\Gamma} F_{t}([x])\chi([x^{-1}y])\overline{\chi([y])}d[x] d[y]\\
& = & \sum_{\chi \in \Gamma^{\perp}} \int_{G/\Gamma} F_{t}([x])\chi([x^{-1}])d[x]\int_{G/K}|\chi([y])|^{2} d[y]\\
& = & \sum_{\chi \in \Gamma^{\perp}} \widehat{F_{t}}(\gamma).\\
& = & \sum_{\chi \in \Gamma^{\perp}} \widehat{f_{t}}(\gamma).\eean \end{proof}

\section{Gaussian and Stable Laws on the Torus}

\subsection {The Gaussian Case} In this section, for $d \in \N$, we take $G = \R^{d}$ and $\Gamma = \Z^{d}$ so that $G/\Gamma$ is the $d$-dimensional torus $\mathbb{T}^{d}$. Note that we also have $\Gamma^{\perp} = \Z^{d}$ in this case, and that for all $f \in L^{1}(\R^{d}), y \in \R^{d}, \widehat{f}(y) = \int_{\Rd}e^{-2\pi i x \cdot y}f(x)dx$. If $h \in {\mathcal S}(\Rd)$, the usual Schwartz space of rapidly decreasing functions, then both (A1) and (A2) are satisfied. The Poisson summation formula (\ref{PSF1}) then takes the familiar form:

\begin{equation} \label{PSFagain}
\sum_{n \in \Z^{d}}h(x + n) = \sum_{n \in \Z^{d}}\widehat{h}(n)e^{2\pi i n \cdot x},
\end{equation}
for each $x \in \R^{d}$. Also in this case, the negative definite function $\eta$ is given by the classical L\'{e}vy-Khintchine formula (see e.g. \cite{Sa}, Theorem 8.1); however, we cannot consider the validity of (\ref{PTF1}) unless the measure $\mu_{t}$ has a sufficiently regular density $f_{t}$  for all $t > 0$. One well--known example where this holds is when $f_{t}$ is Gaussian. In this case $f_{t} \in {\mathcal S}(\Rd)$ for all $t > 0$. We consider the case $d=1$ and $f_{t} \sim N(0, t/2\pi)$. Then for all $ t > 0, x \in \R, f_{t}(x) = \frac{1}{\sqrt {t}}e^{-\frac{\pi x^{2}}{t}}$. Its $\Z$--periodisation is given by $$F_{t}([x]) = \frac{1}{\sqrt{t}}\sum_{n \in \Z}\exp{\left\{- \frac{\pi (n + x)^{2}}{t}\right\}},$$ which is the fundamental solution of the heat equation on the torus,
$$ \frac{\partial F_{t}}{\partial t} = \frac{1}{4\pi}\frac{\partial^{2}F_{t}}{\partial x^{2}}.$$
In this case we have
$$ \theta(t) = \tr(P_{t}) = \sum_{n \in \Z}e^{-t \pi n^{2}},$$
where $\theta$ is the Jacobi theta function (see also \cite{Hey} pp.375--6 where the multivariate case is considered).

As is well--known, (\ref{PTF2}) (or equivalently, (\ref{PSF2})) yields the functional equation
$$ \theta(1/t) = \sqrt{t}\theta(t),$$
for all $t > 0$.

In this case the right hand side of (\ref{intr}) is $\frac{2}{\pi} \sum_{n=1}^{\infty}\frac{1}{n^{2}}$, and so (\ref{intr}) yields the elementary identity
$$ \int_{0}^{\infty}(\theta(t) - 1)dt = \frac{\pi}{3}.$$

\subsection{The Stable Case} Recall that a probability measure $\mu$ on $\Rd$ is stable if given any $a > 0$ there exists $b > 0$ and $c \in \Rd$ so that for all $y \in \Rd$,
$$ \widehat{\mu}(y)^{a} = \widehat{\mu}(by)e^{i c \cdot y}.$$
A measure $\mu$ is rotationally invariant if $\mu(OA) = \mu(A)$ for all $O \in O(d), A \in {\mathcal B}(\Rd)$. It can be shown (see e.g. \cite{Sa} Theorem 14.14, p. 86) that a convolution semigroup $\mu_{t}, t \geq 0)$ is stable and rotationally invariant if and only if there exists $0 < \alpha \leq 2$ and $\sigma > 0$ so that for all $t \geq 0$:

\begin{equation} \label{sta}
\widehat{\mu_{t}}(y) = \exp{\{-t \sigma |y|^{\alpha}\}}.
\end{equation}

The cases $\alpha = 2$ and $\alpha = 1$ are familiar Gaussian and (when $d=1$) Cauchy distributions, respectively. More generally $\mu_{t}$ always has a continuous density $f_{t}$ for $t > 0$, but there may not be a closed form for this. From now on, we assume that $0 < \alpha < 2$. The following useful estimate was obtained in \cite{PT} (see also the discussion in \cite{Sz}, p.1247):

\begin{equation} \label{stdens}
f_{t}(x) \leq \frac{Ct^{-d/\alpha}}{(1 + t^{-1/\alpha}|x|)^{1 + \alpha}},
\end{equation}
for all $t > 0, x \in \R^{d}$, and where $C > 0$ is independent of $x$ and $t$.

Let ${\mathcal M}_{\alpha}(\Rd)$ denote the space of functions of moderate decrease on $\Rd$, so that $g \in {\mathcal M}_{\alpha}(\Rd)$ if and only if there exists $K > 0$ so that for all $x \in \Rd$,
\begin{equation} \label{mod}
|g(x)| \leq \frac{K}{1 + |x|^{1 + \alpha}}.
\end{equation}

We see easily from (\ref{stdens}) that $f_{t} \in {\mathcal M}_{\alpha}(\Rd)$. Furthermore, the Fourier transform is defined on ${\mathcal M}_{\alpha}(\Rd)$, and Fourier inversion holds there (see e.g. \cite{SS}, p.144). In this case (A2) is satisfied, for as shown in equation (4.5) of \cite{App1},

$$ \sum_{n \in \Z^{d}}\widehat{\mu_{t}}(n) = \sum_{n \in \Z^{d}}e^{-\sigma t |n|^{\alpha}} < \infty,$$


\noindent so we can assert that $f_{t}$ projects to a continuous density on $\mathbb{T}^{d}$, by Proposition \ref{Fknown}. Then the argument given in section 3 (or see Theorem 5.4.4 of \cite{App}), ensures that the probabilistic trace formulae (\ref{PTF1}) and (\ref{PTF2}) are both valid.
Summability of the right hand side of (\ref{intr}) requires that $1 < \alpha < 2$. We then have
$$ \int_{0}^{\infty}(F_{t}(0) -1)dt  = 2\sum_{n=1}^{\infty}\frac{1}{n^{\alpha}}.$$
If $0 < \alpha < 1$, then (\ref{intr}) does not hold.

Investigating (A1) is far more tricky in this case. We have

\begin{prop} \label{sumst} For each $x \in \Rd, t > 0, \sum_{n \in \Z^{d}}f_{t}(x + n)$ converges.
\end{prop}

\begin{proof} By (\ref{stdens}) and (\ref{mod}) it is sufficient to show that $$\sum_{n \in \Z^{d}}\frac{1}{1 + (x + n)^{\beta}} < \infty,$$ where $\beta:= 1 + \alpha$. We first consider the case $d = 1$. If $x$ and $n$ are both of the same sign (say they are positive), then
$$ \sum_{n =1}^{\infty}\frac{1}{1 + (x + n)^{\beta}} \leq \sum_{n =1}^{\infty}\frac{1}{n^{\beta}} < \infty,$$
and convergence is even uniform (by the Weierstrass M-test).

If $x$ and $-n$ are of opposite sign (say $x > 0$ and $-n < 0$), then we write $x = \lfloor x \rfloor + \{x\}$, where $\lfloor x \rfloor$ is the integer part of $x$ (i.e. the largest integer that fails to exceed $x$) and
$\{x\} > 0$ is the fractional part of $x$. Then

\bean  \sum_{n =1}^{\infty}\frac{1}{1 + (x - n)^{\beta}} & = & \sum_{n \leq \lfloor x \rfloor}\frac{1}{1 + (x - n)^{\beta}} + \sum_{n > \lfloor x \rfloor}\frac{1}{1 + (x - n)^{\beta}}\\
& \leq & \sum_{n \leq \lfloor x \rfloor}\frac{1}{1 + (x - n)^{\beta}} + \sum_{n > \lfloor x \rfloor}\frac{1}{1 + (\lfloor x \rfloor - n)^{\beta}}\\
& = & \sum_{n \leq \lfloor x \rfloor}\frac{1}{1 + (x - n)^{\beta}} + \sum_{m=1}^{\infty} \frac{1}{1 + m^{\beta}} < \infty. \eean
It follows that for all $x \in \R$, $$\sum_{n \in \Z}\frac{1}{1 + (x + n)^{\beta}} < \infty.$$

When $d > 1$, we may use H\"{o}lder's inequality to deduce that for all $y \in \R^{d}$,

$$ |y|^{\beta} \geq d^{\frac{2 - \beta}{\beta}}\sum_{i=1}^{d}|y_{i}|^{\beta},$$ and so
\bean \sum_{n \in \Z^{d}}\frac{1}{1 + (x + n)^{\beta}} & \leq & \sum_{n \in \Z^{d}}\left(1 + d^{\frac{2 - \beta}{\beta}}\sum_{1=1}^{d}|x_{i} - n_{i}|^{\beta}\right)^{-1}\\
& \leq & \prod_{i=1}^{d}\sum_{n_{i}=1}^{\infty}\frac{1}{1 + d^{\frac{2 - \beta}{\beta}}|x_{i} - n_{i}|^{\beta}} < \infty, \eean
by the one-dimensional result. To obtain the last line in the display we used the elementary inequality
$$ (1 + a_{1})(1 + a_{2}) \cdots (1 + a_{d}) \geq 1 + a_{1}a_{2}\cdots a_{d},$$
for $a_{i} \geq 0, i = 1, 2, \ldots, d$.

\end{proof}

We are unable to verify that the convergence is uniform in Proposition \ref{sumst}, nor that the series converges to an integrable function. So we cannot apply Proposition \ref{dchoice}.
In this case we cannot assert the validity of the Poisson summation formula.
If $d=1$ and $(\mu_{t}, t \geq 0)$ is the Cauchy semigroup, we may verify directly that the Poisson summation formula (\ref{PSF2}) fails. For convenience we take $t = 1$, and write $f:=f_{t}$. Then for all $x \in \R, f(x)= \frac{1}{\pi(1 + x^{2})}$ and $\widehat{f}(x) = e^{-|x|}$. We have
$$ \sum_{n \in \Z}f(n) = \frac{1}{\pi}\left(1 + 2\sum_{n=1}^{\infty}\frac{1}{1 + n^{2}}\right) \leq \frac{1}{\pi}\left(1 + \frac{\pi^{2}}{3}\right) = 1.3655\ldots.$$
However
$$ \sum_{n \in \Z}\widehat{f}(n) = 1 + 2\sum_{n=1}^{\infty}e^{-n} = 1 + \frac{2}{e-1} = 2.1639\ldots.$$



\section{Semistable Semigroups on the $p$-Adics and the Ad\`{e}les}

In this section we will work extensively with $p$--adic analysis and its generalisations to the ad\`{e}le group. Good references for the former are \cite{AKS}, \cite{BrFr} and \cite{Kob}, and for the latter \cite{BrFr} again, \cite{La} and \cite{RV}.

Let ${\mathcal I}$ be a set and $\{(G_{\alpha}, H_{\alpha}), \alpha \in {\mathcal I}\}$ be a collection of pairs wherein $G_{\alpha}$ is a locally compact abelian group and $H_{\alpha}$ is an open subgroup of $G_{\alpha}$. The corresponding {\it restricted direct product} $G$ is the collection of all sets $\{g_{\alpha}, \alpha \in {\mathcal I}\}$ where $g_{\alpha} \in G_{\alpha}$ with $g_{\alpha} \in H_{\alpha}$ for all but finitely many $\alpha \in {\mathcal I}$. $G$ is itself a locally compact abelian group, with the group law defined component--wise. The open sets in $G$ are of the form $\Pi_{\alpha \in {\mathcal I}}U_{\alpha}$, where $U_{\alpha}$ is open in $G_{\alpha}$, with $U_{\alpha} = H_{\alpha}$ for all but finitely many values of $\alpha$.

Let ${\mathcal P}$ denote the set of all prime numbers and ${\mathcal P}^{*}: = \{\infty\} \cup {\mathcal P}$. Recall that for each $p \in {\mathcal P}$, the group $\Q_{p}$ of $p$-adic rationals is the completion of $\Q$ with respect to the metric $d(x,y) = |x - y|_{p}$, where $x, y \in \Q$, and the $p$-adic norm
$ |x|_{p} = p^{-m}$, where $m$ is the unique integer for which $x = \frac{a}{b}p^{m}$, with the integers $a, b$ and $p$ being relatively prime. Define $\Q_{\infty}: = \R$ with its usual metric. For each $p \in {\mathcal P}$, let $\Z_{p}:= \{x \in \Q_{p}: |x|_{p} \leq 1\}$ be the open subgroup of $p$-adic integers. Note that it is also compact. We find it convenient to define $\Z_{\infty} = \Z$. Now take ${\mathcal I}$ to be ${\mathcal P}^{*}$. We obtain the {\it ad\`{e}le group}, to be denoted ${\mathbb A}$, by taking the restricted direct product with $G_{p} = \Q_{p}$ and $H_{p} = \Z_{p}$, for all $p \in {\mathcal P}^{*}$.

The rational numbers $\Q$ are a discrete subgroup of ${\mathbb A}$ with the diagonal embedding $q \rightarrow \{q_{p}, p \in {\mathcal P}^{*}\}$ where $q_{p} = q$ for all $p$. To see this observe that if $q = c/d$ then $d$ is divisible by only finitely many $p \in {\mathcal P}$, and so for all but finitely many $p$, either $p$ divides $c$ but not $d$ (and so $q \in \Z_{p}$), or $p$ divides neither $c$ nor $d$ in which case $|q|_{p} = 1$. It can be shown (see e.g. Chapter 7, Theorem 2, p.139 of \cite{La}) that ${\mathbb A}/\Q$ is compact, and so we are in the context of section 1.

Before proceeding further, let us form another restricted direct product over ${\mathcal P}^{*}$. This time, for $p \in {\mathcal P}$, we take $G_{p} = \Q_{p}^{X} = \Q_{p} \setminus \{0\}$, which is the multiplicative group of non-zero $p$-adic numbers, and $H_{p} = U_{p}:=\{x \in \Q_{p}; |x|_{p} = 1\}$, the group of multiplicative units. We take $G_{\infty} = \R \setminus \{0\}$, and $H_{\infty} = \{-1, 1\}$. In this case, our restricted direct product is the {\it id\`{e}le group}, to be denoted ${\mathbb J}$.  The group ${\mathbb J}$ acts on ${\mathbb A}$ by automorphisms, to be precise if $a \in {\mathbb J}$ and $x \in {\mathbb A}$, then $xa = \{x_{p}a_{p}; p \in {\mathcal P}^{*}\}$. We also need the id\`{e}lic norm of $a \in {\mathbb J}$ which is $||a|| = \Pi_{p \in {\mathcal P}^{*}}|a|_{p}$. Note that the product is finite as $|a|_{p} = 1$ for all but finitely many $p$. We will need the fact that if $dx$ denotes Haar measure on $\A$, then $d(ax) = ||a||dx$, for all $a \in \J$.

Recall that every $x \in \Q_{p}$ may be written uniquely as $x = x_{i} + x_{f}$, where $x_{i} \in Z_{p}$, and either $x_{f} = 0$ or $|x_{f}|_{p} > 1$. We call $x_{f}$, the fractional part of $x$, and write $[x] = x_{f}$. We have $\widehat{\Q_{p}} = \Q_{p}$, and  every character of $\Q_{p}$ is of the form $\chi_{y}$, where $y \in \Q_{p}$ and $\chi_{y}(x) = e^{2 \pi i [xy]}$ for each $x \in \Q_{p}$. Similarly, we have $\widehat{\A} = \A$, and the characters of $\A$ are all of the form $\chi_{y}$, where $y \in \A$ and for all $x \in \A$,

\begin{equation} \label{charad}
\chi_{y}(x) = \Pi_{p \in {\mathcal P}^{*}}\chi_{y_{p}}(x_{p}).
\end{equation}
Again we note that $\chi_{y_{p}}(x_{p}) = 1$ for all but finitely many $p$.

If $f: \A \rightarrow \C$ and $a \in \J$, we define $f_{a}:\A \rightarrow \C$ by $f_{a}(x): = ||a||f(ax)$, for all $x \in \A$.

\begin{prop} \label{trans} If $f$ is a probability density function defined on $\A$ and $a \in \J$, then $f_{a}$ is a probability density function on $\A$.
\end{prop}

\begin{proof} \bean \int_{\A}f_{a}(x)dx & = & ||a||\int_{\A}f(ax)dx \\
& = & \int_{A}f(ax)d(ax) = \int_{\A}f(x)dx = 1. \eean \end{proof}

If (A1) and (A2) hold, then we have the Poisson summation formula from section 1. In fact the following form is often called the {\it Riemann-Roch theorem} in this context, as the well--known formula of that name in algebraic geometry may be derived from it (see section 7.2. of \cite{RV}).

\begin{theorem} \label{RR} If $f$ is continuous, $a \in \J$ and (A1) and (A2) hold for $f_{a}$, then for all $x \in \A$,

$$ \sum_{r \in \Q}f(a(r + x)) = \frac{1}{||a||}\sum_{r \in \Q}\widehat{f}(a^{-1}r)\chi_{r}(x),$$
and when $x = 0$,

$$ \sum_{r \in \Q}f(ar) = \frac{1}{||a||}\sum_{r \in \Q}\widehat{f}(a^{-1}r).$$

\end{theorem}

\begin{proof} This follows from Theorem \ref{PSF}, and then fact that for $y \in \A$,
\bean \frac{1}{||a||}\widehat{f_{a}}(y) & = & \int_{\A}f(ax)\overline{\chi_{y}(x)}dx\\
& = & \int_{\A}f(ax) e^{-2\pi i [xy]}dx\\
& = & \int_{\A}f(x)e^{-2\pi i [xa^{-1}y]}da^{-1}x\\
& = & \frac{1}{||a||}\widehat{f}(a^{-1}y). \eean \end{proof}

To connect with probability theory, our goal (which will not be realised in this paper) should be to

\begin{enumerate}

\item Find some interesting examples of infinitely divisible probability measures and/or convolution semigroups that are defined on the ad\`{e}le group, which have densities, and which satisfy (A1) and (A2).

\item To make contact with the ``Riemann-Roch theorem'', we would want the above to continue to hold when we transform the pdf by an id\`{e}le, as in Proposition \ref{trans}.

\end{enumerate}

To address (1), we first work on the $p$-adic group $\Q_{p}$, for some prime $p$. Following \cite{AKS}, p.24, we say that a mapping $f: O \rightarrow \C$, where $O$ is an open set in $\Q_{p}$, is {\it locally constant} if given any $x \in O$ there exists $l(x) \in \Z$ so that $f(x + y) = f(x)$ for all $y \in \Q_{p}$ for which $|x - y|_{p} \leq p^{l(x)}$ (so that $y$ is in the closed ball of radius $p^{l(x)}$ about $x$). Clearly every locally constant function is continuous. The {\it Bruhat-Schwartz space} of test functions ${\mathcal D}(\Q_{p})$ is the set of all locally constant functions that have compact support. A nice example of an element of ${\mathcal D}(\Q_{p})$ is the function $\gamma _{p} = {\bf 1}_{\Z_{p}}$, the indicator function of the $p$-adic integers. From an analytic perspective, it is a good candidate to be a ``$p$-adic Gaussian'' as it is equal to its own Fourier transform. It is also a pdf, if we scale Haar measure on $\Q_{p}$ in the usual way, so that $\Z_{p}$ has total mass one (and this will be the case from now on). It is certainly infinitely divisible, but also idempotent, and so can be embedded in an i-convolution semigroup of probability measures in which every measure has density $\gamma _{p}$. It is worth pointing out that since $\Q_{p}$ is totally disconnected, an infinitely divisible measure on $\Q_{p}$ can only have a L\'{e}vy-Khintchine formula of pure jump type, as shown in Remark 1 (after Corollary 7.1) of \cite{Par} p.109, and also Proposition 1 of \cite{Ev}; in particular, there are no Gaussians in the probabilistic sense outside the idempotent class.

We now describe the corresponding space of Bruhat-Schwartz test functions on $\A$. This is the set ${\mathcal D}(\A)$ of all mappings $\phi: \A \rightarrow \C$ for which for all $x \in \A$, with $x = (x_{\infty}, x_{2}, x_{3}, x_{5}, \ldots)$:
\begin{equation} \label{bs}
 \phi(x) = \phi_{\infty}(x_{\infty})\Pi_{p \in {\mathcal P}}\phi_{p}(x_{p}),
\end{equation}
where
\begin{enumerate}
\item  $\phi_{\infty}: \R \rightarrow \C$  with $\max_{n \in \Z_{+}}\sup_{x \in \R}|x|^{n}|\phi_{\infty}(x_{\infty})| < \infty,$
\item $\phi_{p} \in {\mathcal D}(\Q_{p})$ for all $p \in {\mathcal P}$,
\item $\phi_{p} = \gamma_{p}$ for all but finitely many $p \in {\mathcal P}$.
\end{enumerate}

Each $\phi \in {\mathcal D}(\A)$ is continuous. It is shown in \cite{RV} p.261, that Theorem \ref{RR} holds when $f \in {\mathcal D}(\A)$. Since such an $f$ is also integrable, we then have a large supply of probability density functions for which Theorem \ref{RR} holds. But apart from the trivial Gaussian case discussed above, there is no obvious connection to more sophisticated probabilistic structure.

We now turn our attention to convolution semigroups on the $p$-adics. At the general level, we remark that there has been work on the solution of the embedding problem of realising an infinitely divisible measure as an element of a convolution semigroup, see \cite{Shah1, Shah2}. We are interested in concrete examples, and as pointed out in the introduction, there has been a substantial amount of work on constructing rotationally invariant semistable processes, mainly from the point of view of solving Kolmogorov's equation to find the transition probabilities (see e.g. \cite{AK2, Yas1}). On the other hand, the emphasis in \cite{HS} pp. 504-6 is on Fourier analysis (see also \cite{Yas2}). In particular in \cite{HS} p.506, it is shown that for any $C, \gamma > 0$, there exists a rotationally invariant, strictly-operator semistable convolution semigroup $(\rho_{t}^{(p)}, t \geq 0)$ of probability measures on $\Q_{p}$ such that, for all $t \geq 0, y \in \Q_{p}$,

\begin{equation} \label{semistable}
\widehat{\rho_{t}^{(p)}}(\chi_{y}) = \exp{\{-Ct|y|_{p}^{\gamma}\}}
\end{equation}

\noindent (see \cite{Yas1,Yas2,DMFT} for related work).

In this context a probability measure $\lambda$ on $\Q_{p}$ is rotationally invariant if $\lambda(uA) = \lambda(A)$ for all $u \in U_{p}, A \in {\mathcal B}(\Q_{p})$ and a convolution semigroup $(\mu_{t}, t \geq 0)$ is strictly-operator semistable if there exists a continuous homomorphism $\delta$ of $\Q_{p}$ and $\beta \in (0, 1)$ so that for all $t \geq 0, \delta\mu_{t} = \mu_{\beta t}$. In the case we are considering we have $\delta x = px$ for all $x \in \Q_{p}$, and $\beta = p^{-\gamma}$. We assume from now on that $(\mu_{t}, t \geq 0)$ is non-trivial, in that for all $t >0, \mu_{t} \neq \delta_{0}$.

We will first show that the measures $\rho_{t}$ have a continuous density for all $t > 0$. We show that the Fourier transform is integrable, and for ease of notation we take $t = 1/C$:

\begin{theorem} \label{maxstab} For all $\gamma > 0$, \begin{enumerate}

\item $\int_{\Z_{p}}e^{-|y|_{p}^{\gamma}}dy = \ds\frac{p-1}{p}\sum_{n=0}^{\infty}\frac{e^{-p^{-n\gamma}}}{p^{n}}.$
\item $\int_{\Q_{p}}e^{-|y|_{p}^{\gamma}}dy = (p-1)\sum_{n \in \Z}\psi(p,n)p^{n}e^{-p^{n\gamma}},$

where
$\psi(p,n): = \left\{\begin{array}{c c} 1/p & \mbox{if}~n < 0\\
e^{-p\gamma}& \mbox{if}~n > 0\\
1/p + e^{-p\gamma}& \mbox{if}~n = 0\end{array} \right. $

\end{enumerate}

\end{theorem}

\begin{proof} We follow the method of calculation of integrals given in \cite{BrFr} pp. 12-14.

\begin{enumerate}

\item Write $\Z_{p} = \bigcup_{k = 0}^{p-1}C_{k}$, where $C_{k}:= \{x \in \Z_{p}; x = k + \sum_{i=1}^{\infty}a_{i}p^{i}, a_{i} = 0, \ldots, p-1\}$. Since, by translation invariance, the Haar measure of $C_{k}$ takes the same value for all $k = 0, \ldots, p-1$, we deduce that
    $$\int_{\Z_{p}}e^{-|y|_{p}^{\gamma}}dy = \ds\frac{p-1}{p}e^{-1} + \int_{C_{0}}e^{-|y|_{p}^{\gamma}}dy.$$
   We now write $C_{0} = \bigcup_{l = 0}^{p-1}C_{0,l}$, where $C_{0,l}:= \{x \in \Z_{p}; x = lp + \sum_{i=1}^{\infty}a_{i}p^{i+1}, a_{i} = 0, \ldots, p-1\}$. By a similar argument to the above, we find that
    $$\int_{C_{0}}e^{-|y|_{p}^{\gamma}}dy = \frac{p-1}{p^{2}}e^{-p^{-\gamma}} + \int_{C_{0,0}}e^{-|y|_{p}^{\gamma}}dy.$$
    Iterating this argument leads to
    $$\int_{\Z_{p}}e^{-|y|_{p}^{\gamma}}dy = \frac{p-1}{p}e^{-1} + \frac{p-1}{p^{2}}e^{-p^{-\gamma}} + \frac{p-1}{p^{3}}e^{-p^{-2\gamma}} + \cdots,$$
    and the result follows easily.

    \item Write $c_{p, \gamma}: = \int_{\Z_{p}}e^{-|y|_{p}^{\gamma}}dy$ to ease the burden of notation. Let $K^{(-1)}: = \{x \in \Q_{p}, |x|_{p} \leq p\} = \bigcup_{k = 0}^{p-1}K^{(-1)}_{m}$, where $K^{(-1)}_{m}:= \{x \in K^{(-1)}; x = mp^{-1} + y, y \in \Z_{p}\}$. By translation invariance, we have for all $m = 1, \ldots, p-1, \int_{K^{(-1)}_{m}}e^{-|y|_{p}^{\gamma}}dy = e^{-p^{\gamma}}$ and so
        $$ \int_{K^{(-1)}}e^{-|y|_{p}^{\gamma}}dy = (p-1)e^{-p^{\gamma}} + c_{p, \gamma}.$$
        Now for all $\nN$ define $K^{(-n)}: = \{x \in \Q_{p}, |x|_{p} \leq p^{n}\}$, then by a similar argument to that given above we obtain
        \bean \int_{K^{(-n)}}e^{-|y|_{p}^{\gamma}}dy & = & (p-1)[p^{n-1}e^{-p^{n \gamma}} + p^{n-2}e^{-p^{(n-1) \gamma}} + \cdots + e^{-p^{\gamma}}] + c_{p, \gamma} \\
        & = & (p-1)e^{-p^{\gamma}}\sum_{r=0}^{n-1}p^{r}e^{-p^{r\gamma}} + c_{p, \gamma}\eean
        But by the monotone convergence theorem
        $$ \int_{\Q_{p}}e^{-|y|_{p}^{\gamma}}dy = \lim_{n \rightarrow \infty}\int_{K^{(-n)}}e^{-|y|_{p}^{\gamma}}dy = (p-1)e^{-p^{\gamma}}\sum_{n=0}^{\infty}p^{n}e^{-p^{n\gamma}} + c_{p, \gamma}. $$

\end{enumerate}
\end{proof}

Note that both series obtained in Theorem \ref{maxstab} converge; the convergence of the first is obvious, and that of the second follows from using Cauchy's root test. In fact we could also have deduced the finiteness of the integrals in Theorem \ref{maxstab} from the more general result of Proposition 3.1 in \cite{Yas2}. The proof therein uses a more probabilistic approach, but does not compute the integrals explicitly.



It is now easy to see that $\mu_{t}$ has a continuous density for all $t > 0$, by essentially the same method as used in Proposition \ref{Fknown}. Define for all $x \in \Q_{p}$,
\begin{equation} \label{dens}
 f_{t}^{(p)}(x) = \int_{\Q_{p}}\chi_{y}(x)e^{-Ct|y|_{p}^{\gamma}}dy.
\end{equation}
Then $f_{t}$ is continuous by standard use of dominated convergence. But $f_{t}$ is the density of $\mu_{t}$ by uniqueness of Fourier transforms.

We will obtain a series expansion for the density $f_{t}$.  To carry this out we will need the Gel'fand--Graev gamma function (see \cite{BrFr} and \cite{GGP} pp.144-51) $\Gamma_{p}: \C \rightarrow \C$ defined by
\begin{equation} \label{gamma1}
\Gamma_{p}(s) = \int_{\Q_{p}}\chi_{1}(x)|x|_{p}^{s-1}dx,
\end{equation}
where $s \in \C$. It is shown in \cite{BrFr} that

\begin{equation} \label{gamma2}
\Gamma_{p}(s) =\ds\frac{1 - p^{s-1}}{1 - p^{-s}}.
\end{equation}

In the sequel, we will consider the case $s \in \R$ with $s > 0$, in which case we have the easy estimate:
\begin{equation} \label{gamma3}
|\Gamma_{p}(s)| \leq p^{s-1}.
\end{equation}

\begin{theorem} \label{ldstab} For each $t > 0, x \in \Q_{p} \setminus \{0\}$,
$$ f_{t}^{(p)}(x) = \sum_{n=0}^{\infty}\frac{(-1)^{n}}{n!}\frac{(Ct)^{n\gamma}}{|x|_{p}^{n\gamma + 1}}\Gamma(n\gamma + 1).$$
\end{theorem}

\begin{proof} By (\ref{dens}) we have the following, where we've chosen $t = 1/C$ for convenience,
\bean f_{t}^{(p)}(x) & = & \int_{\Q_{p}} \chi_{1}(xy) e^{-|y|_{p}^{\gamma}}dy \\
& = & \int_{\Q_{p}} \chi_{1}(xy)\sum_{n=0}^{\infty}\frac{(-1)^{n}}{n!}|y|_{p}^{n\gamma}dy. \eean

However, for $x \neq 0$,
\bean  \sum_{n=0}^{\infty}\frac{(-1)^{n}}{n!} \int_{\Q_{p}} \chi_{1}(xy)|y|_{p}^{n\gamma}dy & = & \sum_{n=0}^{\infty}\frac{(-1)^{n}}{n!}\frac{1}{|x|_{p}^{n\gamma + 1}} \int_{\Q_{p}} \chi_{1}(y)|y|_{p}^{n\gamma}dy \\
& = & \sum_{n=0}^{\infty}\frac{(-1)^{n}}{n!}\frac{1}{|x|_{p}^{n\gamma + 1}}\Gamma(n\gamma + 1). \eean
Now using (\ref{gamma3}), we have
\bean  \sum_{n=0}^{\infty}\frac{1}{n!}\frac{1}{|x|_{p}^{n\gamma + 1}}|\Gamma(n\gamma + 1)| & \leq & \frac{1}{|x|_{p}}\sum_{n=0}^{\infty}\frac{1}{n!}\left(\frac{p}{|x|_{p}}\right)^{n\gamma}\\
& = & \frac{1}{|x|_{p}}\exp{\left\{\frac{p^{\gamma}}{|x|_{p}^{\gamma}}\right\}}, \eean and the result follows by Fubini's theorem.

\end{proof}

It is interesting to compare the series expansions just established for the $p$-adic semistable laws to those obtained on the real line for $\alpha$-stable laws -- see \cite{Fe2}, p.548-9.

We consider probability measures $\mu$ on the ad\`{e}les $\A$ which are defined by analogy with the Bruhat-Schwartz test functions. To be precise, suppose that for each $p \in {\mathcal P}^{*}$ we are given a probability measure $\mu_{p}$ defined on $\Q_{p}$, then we define a positive linear functional $\mu$ on ${\mathcal D}(\A)$ as follows:
\begin{equation} \label{prodmeas}
\mu(\phi) = \mu_{\infty}(\phi_{\infty})\Pi_{p \in {\mathcal P}}\mu_{p}(\phi_{p}),
\end{equation}
for all $\phi \in {\mathcal D}(\A)$ having the form (\ref{bs}). Here we have also assumed that for all but finitely many $p \in \mathcal{P}, \mu_{p}(\phi_{p}) = \int_{\Q_{p}}\phi_{p}(x_{p})\gamma_{p}(x_{p})dx_{p}$. We can in fact extend the domain of definition of $\mu$ in (\ref{prodmeas}) to a wider class of functions on $\A$, in particular to those $\phi$ having the product form (\ref{bs}) wherein $\phi_{p}$ is a bounded measurable function for all $p \in {\mathcal P}^{*}$, with only finitely may of the $\phi_{p}$'s having (supremum) norm exceeding one.  It is then clear that $\mu$ induces a probability measure on $\A$. If $\mu_{p}$ has a pdf $f_{p}$ for all $p \in {\mathcal P}^{*}$, then it is easy to see that $\mu$ has a density $f$ where
\begin{equation} \label{densad}
f(x) = f_{\infty}(x_{\infty})\Pi_{p \in {\mathcal P}}f_{p}(x_{p}),
\end{equation}
where $x = (x_{\infty}, x_{2}, x_{3}, \ldots) \in \A$, and from the definition of $\mu, f_{p} = \gamma_{p}$ for all but finitely many values of $p$. By \cite{RV}, Proposition 5-6 (ii), p.186, the function $f$ is continuous if $f_{p}$ is continuous for all $p \in {\mathcal P}^{*}$ and furthermore, for all $h \in L^{1}(\A, \C)$  of the form (\ref{densad}),
\begin{equation} \label{l1}
 \int_{\A}h(x)dx = \prod_{p \in {\mathcal P}^{*}} \int_{\Q_{p}}h_{p}(x_{p})dx_{p},
 \end{equation}
 whenever the right hand side is finite. Hence for all $y \in \A$, by (\ref{charad}), we have
$$ \widehat{f}(y) = \widehat{f_{\infty}}(y_{\infty})\Pi_{p \in {\mathcal P}}\widehat{f_{p}}(y_{p}).$$

We consider a family of probability measures $(\mu_{t}, t \geq 0)$ on $\A$, where for each $t \geq 0, \mu_{t}$ is of the form (\ref{prodmeas}). For $t > 0$, the finitely many non-Gaussian $\mu_{t}^{(p)}$'s will all be $\gamma$-semistable measures on $\Q_{p}$ (which we previously denoted as $\rho_{t}^{(p)}$), while $\mu_{t}^{\infty}$ is a rotationally invariant $\alpha$-stable measure with $0 < \alpha \leq 2$. For convenience we will take the indices $\gamma$ and $\alpha$,  and the constants $\sigma$ and $C$ to be the same for each of these measures. Then $\mu_{t}$ has a density $f_{t}$ of the form (\ref{densad}) for all $t > 0$.   We will also fix a finite set $S = \{p_{i_{1}}, \ldots, p_{i_{N}}\} \subset {\mathcal P}$ so that for all $t > 0, j = 1, \ldots, N, f_{t}^{(p_{i_{j}})} \neq \gamma_{p_{i_{j}}}$. Using (\ref{prodmeas}), we see that $(\mu_{t}, t \geq 0)$ is an i-convolution semigroup. To check the validity of (\ref{PTF2}), we must verify (A2) and investigate the behaviour of $\sum_{r \in \Q}\widehat{\mu_{t}}(r)$, for $t > 0$.

Now for $p \in S^{c}, \gamma_{p}(r) \neq 0$ if and only if $p \in \Z_{p}$. It follows that for non-vanishing $\widehat{\mu_{t}}(r)$, we must have $r = a/b$, where $a \in \Z, b \in \N$ and the prime factorisation of $|a|$ features only primes in $S^{c}$, while that of $b$ involves only primes in $S$. Let $D:= \{r =a/b; \gamma_{p}(r) \neq 0~\mbox{for all}~p \in S^{c}\}$. Then
$$ \sum_{r \in \Q}\widehat{\mu_{t}}(r) = \sum_{r \in D}\widehat{f_{t}^{(\infty)}}(r)\widehat{f_{t}^{(p_{i_{1}})}}(r)\cdots \widehat{f_{t}^{(p_{i_{N}})}}(r).$$
Now let $D = B \cup B^{c}$, where $B:= \{r = a/b \in D; b = 1\}$. If $r = a/b \in B^{c}$, we must have $b = p_{i_{1}}^{m_{1}}p_{i_{2}}^{m_{2}}\cdots p_{i_{N}}^{m_{N}}$ where $m_{1}, m_{2}, \ldots, m_{N} \in \Z_{+}$ and at least one of these is non-zero. Then we have
\bean \sum_{r \in \Q}\widehat{\mu_{t}}(r) & = & \sum_{r \in B}\widehat{\mu_{t}}(r) + \sum_{r \in B^{c}}\widehat{\mu_{t}}(r)\\
& \leq & \sum_{r \in B}\widehat{f_{t}^{(\infty)}}(r) + \sum_{r \in B^{c}}\widehat{f_{t}^{(\infty)}}(r)\widehat{f_{t}^{(p_{i_{1}})}}(r)\cdots \widehat{f_{t}^{(p_{i_{N}})}}(r). \eean
We have
$$ \sum_{r \in B}\widehat{f_{t}^{(\infty)}}(r) \leq \sum_{n \in \Z}e^{-\sigma t |n|^{\alpha}} < \infty.$$
 Now let $p_{0}:= \min(S)$ and $A \in \N$ be coprime to $p_{0}$, and define $E \subset B^{c}$ by $E:=\{r = a/b \in B^{c}, a = A, b= p_{0}^{n}q, n \in \N\}$. Then there exists $K > 0$ so that
 \bean \sum_{r \in B^{c}}\widehat{\mu_{t}}(r)  & \geq & \sum_{r \in E}\widehat{\mu_{t}}(r) \\
 & = & K \sum_{n=1}^{\infty}e^{-\sigma t (A/p_{0}^{n})^{\alpha}}\sum_{m=1}^{\infty}e^{-C t p_{0}^{m \gamma}}.\eean

Now $\sum_{n=1}^{\infty}e^{-\sigma t (A/p_{0}^{n})^{\alpha}} = \infty$ as $\lim_{n \rightarrow \infty}e^{- \beta/y^{n}} = 1$ for all $\beta > 0, y > 1$. But $0 < \sum_{m=1}^{\infty}e^{-C t p_{0}^{m \gamma}} < \infty$.

We conclude that $\sum_{r \in \Q}\widehat{\mu_{t}}(r) = \infty$, and so the probabilistic trace formulae are not valid in this case. Indeed we have just shown that $\tr(P_{t}) = \infty$ for all $t > 0$. In then follows, by a slight extension of Theorem 5.4.4 in \cite{App} p.144, that the projection of $\mu_{t}$ to ${\mathbb A}/\Q$  cannot have a square-integrable density for any $t > 0$.

It would be interesting to clarify the relationship between our convolution semigroup on $\A$ and the Markov process constructed in \cite{Yas}, which has such interesting connections with number theory (see also \cite{Urb, Yas4}). In fact the processes constructed in those papers have a very similar structure to ours, but they are not of ``Bruhat-Schwartz type'', in that semistable activity is manifest in every $p$-adic component of the ad\`{e}les.

\vspace{5pt}

{\it Acknowledgement.} I am very grateful to Nick Bingham for reading a draft of the manuscript, and providing some very helpful comments.

\end{document}